\documentclass[12pt]{article}

\usepackage{amssymb}
\usepackage{amsmath}
\usepackage{amsthm}
\usepackage[T1]{fontenc}
\usepackage[utf8]{inputenc}
\usepackage[english]{babel}
\usepackage[section]{algorithm}
\usepackage{algorithmic}
\usepackage{graphicx} 
\usepackage[normalsize]{subfigure}
\usepackage{caption}
\usepackage{url}
\usepackage[normalsize]{subfigure} 
\usepackage{tikz-cd}

\newtheorem{theorem}{Theorem}[section]
\newtheorem{definition}[theorem]{Definition}

\newtheorem{lemma}[theorem]{Lemma}

\newtheorem{example}[theorem]{Example}

\usepackage{geometry}
\usepackage{float}

\begin{document}

	\title{Reduction of Simplicial Complex by Relation and Dowker Complex}
	\author{Dominic Desjardins Côté}
	
	\maketitle
	\bibliographystyle{plain}	
	\begin{abstract}
		We show a new reduction method on a simplicial complex. This reduction works well with relations and Dowker complexes. The idea is to add a dummy vertex $ z $ to the simplicial complex $K$. We add the simplicial cone $ z * L $ to $K$  where $ L$ is the union of stars from a set of vertices. If $ L $ is contractible, then we can apply the Gluing theorem to glue $ z * L $ to $K$ to obtain $K'$. Finally, we strong collapse each vertex of $L$ in $K'$ to obtain $K''$. If the conditions are satisfied, then $K$, $K'$ and $K''$ are homotopically equivalent.
		
		 This trick can be adapted to relation with the associated Dowker complex $K_R$. This notation help to simplify various computations. Relations are simple data structures, and they are represented by binary matrices. This method of reduction with relation is versatile and it can be used on different structures such as simplicial complexes, convex polytopal complexes and covers of topological spaces that satisfy the Nerve Theorem. We develop an algorithm based on the reduction step. Let $n$ be the number of vertices of $K$.
		 
		  We have $ O(n^2) $ subcomplexes $ L$ to verify contractibility. This verification of $ L $ is costly with $ O(d \epsilon (n^2 + m^2)) $ where $d$ is the dimension of $L$, $m$ the number of toplexes in $L$, $n$ the number of vertices in $L$ and $ \epsilon $ the maximal number of toplexes adjacent to a vertex in $L$. But, $L$ is often a small simplicial complex. If $L$ is contractible, then we apply a clean-up method on some columns that takes $ O(d m^2) $. Finally, we show the efficiency of the reduction algorithm on several experimental results.
	\end{abstract}		
	
	\setlength{\parskip}{1em}	
	
	\section{Introduction}
		
		An important aspect in topology is to compute the homology of a topological space $X$. These homological features give a nice description of our topological spaces. In computational topological, we discretize our space into a simplicial complex $S$ which have the same homology of $X$.
		
		To compute homology of $ S $ in a classic way, one uses the Smith normal form which can be long to compute. A way to make it faster is to reduce the number of simplices with methods that are homological invariants. There exists multiple methods to reduce the number of simplices. This list of references is not exhaustive but it covers various approaches \cite{arDisMorTheAlgo, strongColPers, acySubspaceSimp, edgeContractionSimpHomo}. In this paper, we show a novel new way to reduce a simplicial complex $S$. 
		
		The reduction method has two steps. First, we add a dummy vertex $ z $. We choose $ n $ vertices and compute the sub complex $ L = \cup_{i=1}^n \overline{St(x_i)} $. If $ L $ is contractile, then by the Gluing Theorem, we can attach $ z*L $ to $ K $ to obtain $ K' $ where $ *$ is the simplicial joint. For the second step, each vertex of $L$ is now dominated by $z$ in $K'$. Therefore, we can strong collapse each vertex of $L$ to obtain $K''$. If the conditions are satisfied, then we obtain that $ K $, $ K' $ and $ K'' $ are homotopically equivalent. If we choose $ n > 1 $, then this reduction method removes $n$ vertices which are greater than the single vertex we added. Furthermore, it also removes some simplices. This happens in two cases. First, we have that two simplices can become the same one after the reduction. Second, some simplices have their dimension reduced and can become the face of another simplex.
		
		\begin{figure}
  			\center
  			\subfigure[ The simplicial complex $K$. ]{
   				\includegraphics[height=5cm, width=5cm, scale=1.00, angle=0 ]{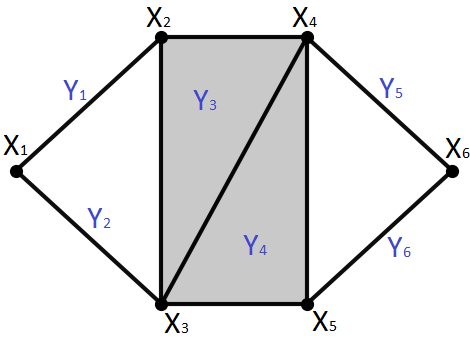}\label{figIntroCmpa}
  			}
  			\,
  			\subfigure[ The simplicial complex $ L  $. This is contractible.  ]{
   				\includegraphics[height=5cm, width=5cm, scale=1.00, angle=0 ]{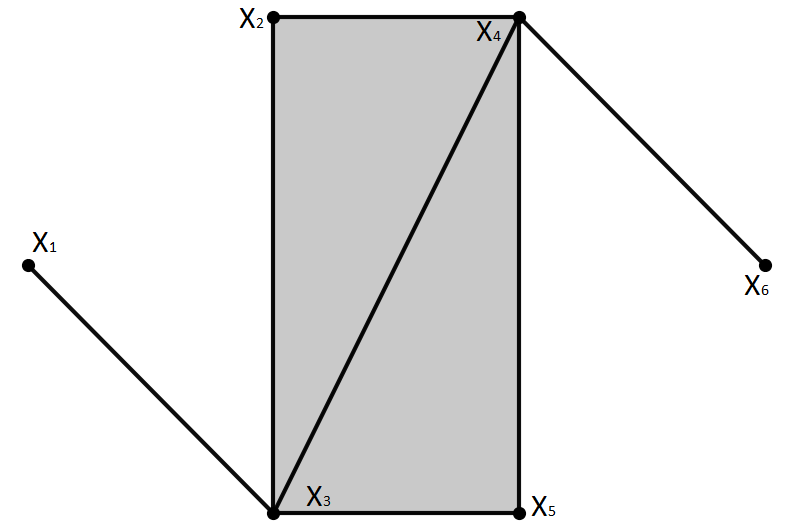}
            }
            \,
  			\subfigure[ This is the result after a reduction step.]{
   				\includegraphics[height=5cm, width=5cm, scale=1.00, angle=0 ]{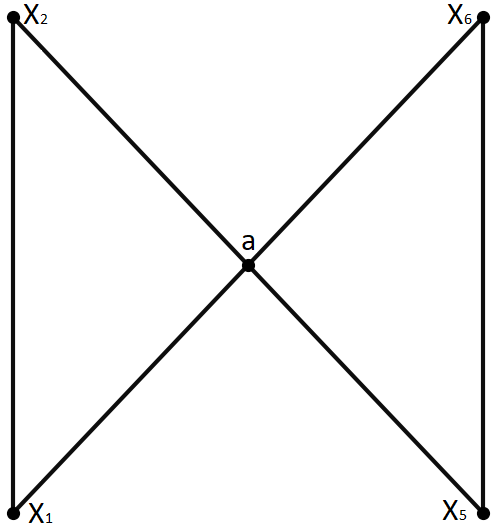}
            }
            \caption{This is an example of the reduction step on simplicial $K$ with $ L = \overline{St(x_3)} \cup \overline{St(x_4)}$. Between the Figure (a) and the Figure (c), we glue $ a * L $ to $K$ and we strong collapse $ x_3 $ and $ x_4 $. }\label{figIntro}
		\end{figure}
		
		For the computation, we argue that relation and Dowker complexes are a well-suited data structure for our reduction method. 
		 
		Let us remind some definitions. From a relation $R \subset X \times Y$, we can define two simplicial complexes $K_R $ and $ L_R $. Let $ \sigma = [x_1, x_2, \ldots, x_n] \in K_R $ if and only if there exists an $y$ such that $ (x_i, y) \in R $ for all $i$. By analogous construction, $ \tau = [ y_1, y_2, \ldots, y_n ] \in L_R $ if and only if there exists an $ x $ such that $ (x, y_j) \in R $ for all $ j$ . They are called Dowker complexes. In litterature, they are also called Witness complexes \cite{rGhristEAT}. From the Dowker Theorem \cite{bjorner1995}, $ \vert K_R \vert $ and $ \vert L_R \vert $ are homotopy equivalent. The Dowker Theorem has multiple applications and have great success in the theory of formal concept \cite{boFormalConcept} and Q-analysis in social studies \cite{arQanalAtkin, arQanalAppSoc}.
			
		An advantage of a relation is $ R $ can be represented by a binary matrix. From a simplicial complex $ S $, we can easily construct a relation such as $ S = K_R $. Let $ X $ be the set of vertices and $ Y $ be the set of toplexes of $S$. We define the relation $ R \subset X \times Y $ such that $ (x, y) \in R $ if and only if $ x \in y $.
		
		The operations that we need to do are quite easy except one. To compute $ L = \cup_{i = 1}^n \overline{St(x_i}) $, we need to compute the following submatrix : $R_{St} \subset X \times (\cup_{i=1}^n R(x_i)) $ such that $ R_{St}(x) = R(x) \cap (\cup_{i=1}^n R(x_i)) $. For gluing $ z *L $, we add a new row $ z $ to $R$ where $ R(z) = \cup_{i=1}^n R(x_i) $. For strong collapse on $ x_i $, we only need to remove the row associated to $ x_i $. Finally, the hardest operation is to verify that $ L $ is homotopically trivial. For our problem, we use the algorithm of \cite{strongColPers}. Their methods use strong collapse on binary matrix. If we obtained a $ 1 \times 1$ matrix from the algorithm of \cite{strongColPers}, then $L$ is strong collapsible which implies that is also contractible. We choose this algorithm because it can be done on binary matrices without computing Dowker complexes.
		
		We developed an algorithm from our reduction method with two vertices. We have $ O(n^2) $ different $L$ to test where $ n $ is number of vertices in $K$. To verify $L$ that is contractible $ O(d \epsilon(n^2+m^2)) $ where $ d $ is the dimension of $L$, $ n $ the number of vertices in $L$, $ m $ the number of toplexes in $L$ and $ \epsilon $ the maximal number of toplexes containing an $x \in X$ in $L$. Finally, we apply a clean-up method on columns to reduce the memory. The cost of this is  $ O(d m^2) $. We suppose that the other operations are done in constant time.
		
		In section \ref{secPrior}, we discuss about related works. In section \ref{secPrelim}, we recall definitions and results about simplicial complex, Dowker complex and the Gluing Theorem. In section \ref{secRedStep}, we show the reduction step. We also define the notation of relation and their associated Dowker complexes. In section \ref{secData}, we construct relations on different topological structures. They are simplicial complex, convex polytopal complex and covering of a topological space. In section \ref{secRedAlgo}, we define the algorithm on our reduction methods and we discuss about time complexity. We go into detail in every step and things to avoid which can slow down the computing time. In section \ref{secExpRes}, we show some experimental results and we discuss about the efficiency of our method.		
		
	\section{Prior Works}\label{secPrior}
	
	There already exists multiple reduction methods for a simplicial complex. We will cover some methods related to us. 
	
	The first similar method is \cite{strongColPers}. First, we use the same data structure which is an adjacency matrix of vertices and toplexes. This is the same relation that we defined earlier. Authors use strong collapse and strong homotopy type \cite{strongHomType}. A row $x_i$ is dominated by a row $x_j$, if $ x_i $ as one at column $k$ then $ x_j $ also as a one at the same column. If this is the case, then we can strong collapse $ x_i $ by removing $x_i$ and all simplices containning $ x_i $. When we have done all possible strong collapses on the rows, we repeat the same process for columns. The resulted simplicial complex has the same strong homotopy type as the initial complex \cite{strongHomType}. We repeat until we cannot reduce the simplicial complex anymore. The method of \cite{strongColPers} is a special case of our method. This happens when $ n = 2 $ and $ \overline{St({x_1})} \subset \overline{St(x_2)} $. From Figure \ref{figIntroCmpa}, $K$ cannot be reduced any more by strong collapse. But with our method, we can still reduce $K$.

	The second method is by edge collapse \cite{edgeContractionSimpHomo, topoPresEdgeCont, boHighDimCollapse, inProcHighDimEffDataStruc}. This is also called edge contraction. Let $ a $ and $b$ be vertices and $ ab $ an edge of $ K $. The edge $ ab $ satisfies the Link condition if $ Lk(ab) = Lk(a) \cap Lk(b) $. We define the following simplicial map $ r : K \to K \setminus \{ a \} $ where $ r(a) = b $ and identity for others vertices. If  $ ab $ satisfies the Link condition, then $ r $ preserves the homotopy type \cite{inProcHighDimEffDataStruc}. The edge collapse method is also a special case from our method. This happens when $ n = 2 $, $ x_i \in \overline{St(x_j)} $ and we set the vertex $ z = x_j $. An advantage of edge collapse is that the Link condition is easier to compute than verifying that $L$ is contractible. Our method can reduce it further. This can be seen from simplicial complexes in Figure \ref{figPrio}.
	
		In \cite{boHighDimCollapse}, they used a similar structure in one of their methods. It is called Generalized Indexed data structure with Adjacencies \cite{arAdjacency}. This structure encodes the adjacency matrix of vertices and toplexes, the boundary relation between toplex and vertices, the adjacency matrix between toplexes, and partial information about the star of each vertex. In our method, we do not need all this information. We only need the adjacency matrix of toplexes and vertices which we represent the data by a relation. In \cite{boHighDimCollapse}, they also developed a Link condition by using only toplexes instead of computing each simplex in $ Lk(a) $, $ Lk(b) $ and $ Lk(ab) $. This reduce the computation of edge collapse for high dimension complex.
	
		\begin{figure}
  			\center
  			\subfigure[ The simplicial complex $K$ ]{
   				\includegraphics[height=4.5cm, width=6cm, scale=1.00, angle=0 ]{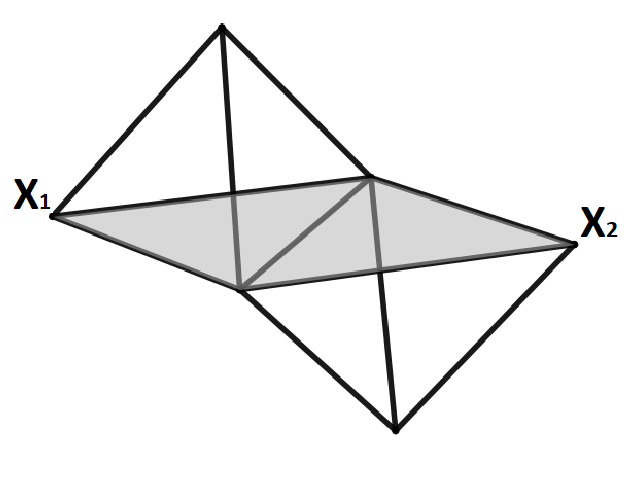}\label{figPriorCmpa}
  			}
  			\,
  			\subfigure[ The reduced simplicial complex $ K'' $  ]{
   				\includegraphics[height=4.5cm, width=6cm, scale=1.00, angle=0 ]{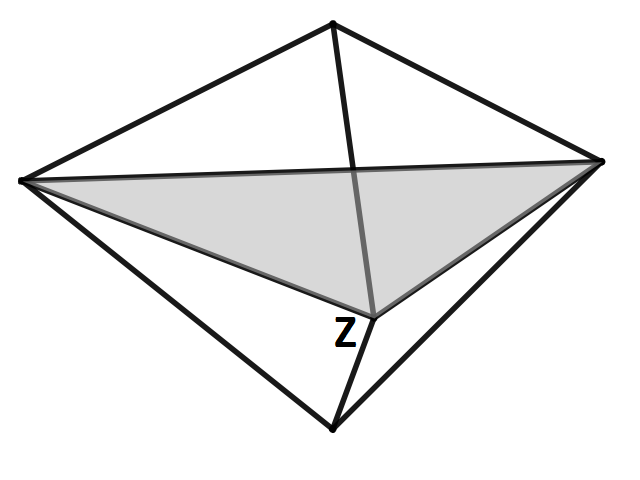}\label{figPriorCmpb}
            }
            \caption{This is an example of the reduction step on simplicial $K$ where we can't apply any edge collapse. Every edge in Figure \ref{figPriorCmpa} does not satisfy the Link condition. But we can still apply our reduction with $ L = \overline{St(x_1)}  \cup \overline{St(x_2)} $ which is contractible. We obtain the reduced simplicial complex in Figure \ref{figPriorCmpb}.}\label{figPrio}
		\end{figure}
	
	Another approach is to reduce by acyclic subspace \cite{acySubspaceCub, acySubspaceSimp}. The method comes from the following observation. There is a Theorem \cite{acySubspaceSimp} which says that if $ A $ is an acyclic subspace of a space $ X $, then $ H_n(X) \cong H_n(X,A) $ for $ n \geq 1 $ and $ H_0(X) \cong \mathbb{Z} \bigoplus H_0(X, A) $. The idea is we suppose that computing the relative homology of $ H_n(X,A) $ is easier than $ H_n(X) $. The effectiveness of this method is by choosing a large set $ A $. We could also repeat the process with another acyclic subspace $ B $ to compute $ H_n(X, A) $\cite{acySubspaceCub}. This simplify the computation of $ H_n(X,A) $. We have a similar problem as \cite{acySubspaceSimp}. We need to verify that a set of simplices is contractible. Authors in \cite{acySubspaceSimp} need an acyclic set and they give different methods : computing the homology directly, a full test and a partial test. We choose a partial test by testing the strong collapsibility of $L$ instead of verifying that $L$ is contractible.
	
	The construction $ z * L $ is also called a simplicial cone. In \cite{thesisConingCol}, different types of cone are constructed on a simplicial complex. The author in \cite{thesisConingCol} used anti-collapse to create a cone on the simplicial complex and  collapse the vertices to reduce the complex. In our method, we use the Gluing Theorem to add a cone and we use strong collapse to remove the vertices which is faster than finding a sequence of collapses and anti-collapses. Similar cone method like ours can be found in \cite{inProcTopPersSimpMap}. But it is used for different purposes.	  	
		
	\section{Preliminaries}\label{secPrelim}	
	
	\subsection{Simplicial Complex}	
	
	An abstract simplicial complex is a set $ K $ that contains finite non-empty sets such as if $  A \in K$, then all subsets of $A$ are also in $K$. Subsets of $ K $ are called simplices and elements of simplices are called vertices. We denote an $n$-simplex by $ [ x_0, x_1, \ldots, x_n ] $ spanned by the set of vertices $ \{ x_0, x_1, \ldots, x_n \} $. All subsets of a simplex are called faces and we denote by $ \sigma \leq \tau $, if $ \sigma $ is a face of $ \tau  $. We say $ L $ is a subcomplex of $ K $ if $ L $ is a simplicial complex and all $  \sigma \in L$ implies that $ \sigma \in K $. We denote $ \vert K \vert $ the geometric realization of an abstract simplicial complex $K$. We define a simplex $ \sigma $ is a toplex if there exists $ \tau $ such that $ \sigma \leq \tau $, then $ \sigma = \tau $. The dimension of a simplex $ \sigma  $ is the number of vertices minus one. The dimension of $K$ is the maximum dimension of its simplices.

	The combinatorial closure of a set of simplices $A$ in $K$ is :
	\begin{equation*}
		\overline{A} := \{ \sigma \in K \mid \exists \tau \in A \text{ such that } \sigma \leq \tau  \}.
	\end{equation*}
	
	The star of a simplex $ \sigma $ is all simplices in $K$ such that $ \sigma $ is a face.	
	\begin{equation*}
		St(\sigma) := \{ \tau \in K \mid \sigma \leq \tau \}.
	\end{equation*}
	
	The star of a simplex is not a subcomplex of $K$. But the closure of a star is a subcomplex of $K$. Let $ A $ be a set of simplices. We define the star of $A$ by $ St(A) := \{ \tau \in K \mid \exists \sigma \in A \text{ such that } \sigma < \tau \} $.	
	
	\begin{definition}
		Let $K$ and $L$ be simplicial complexes. The simplicial joint $ K * L $ is 
		\begin{equation*}
			K*L = K \cup L \cup \{ \sigma \dot{\cup} \tau \mid \sigma \in K, \tau \in L \},
		\end{equation*}				
		where $ \sigma \dot{\cup} \tau $ is the simplex spanned by the vertices of $ \sigma $ and $ \tau $.
	\end{definition}
	
	The simplicial joint of a simplicial complex with a vertex is called a simplicial cone. This is the combinatorial version of the cone of a topological space.	

	Let us define strong collapse \cite{strongHomType}. We say that a vertex $ v_1 $ is dominated by a vertex $v_2$, if there exists a subcomplex $L$ such that $Lk(v_2) $ is the simplicial cone $ v_2 * L $. We can delete the vertex $ v_1 $ from the simplicial complex without changing the homotopy type of $K$. This step of deletion is called strong collapse and we denote $ K \searrow_{SC} K \setminus \{ v_1 \} $. We define a retract map $ r : K \to K \setminus \{ v_1 \} $ such that $ r(v_1) = v_2 $ and the other vertices are the identity. We also have the inclusion $ i : K \setminus \{ v_1 \} \hookrightarrow K $. We have that $ r $ and $ i $ are homotopic. Precisely, $r$ is a strong deformation retract \cite{strongHomType}. If there exists a sequence of strong collapses and strong expansions between $K$ and $K'$, then they have the same strong homotopy type. We say that $K$ is strong collapsible if there exist sequence such that the final complex is a point. In our algorithm, we use strong collapsibility to show that a simplicial complex is contractible.

	\subsection{Relation and Dowker Complex}	

		Let $ X $ and $ Y $ be finite sets and define a relation $ R $ a subset of $ X \times Y $. For a pair $ (x, y) \in R $, we use two other notations $ x R y $ and $ y \in R(x) $. The inverse relation is $ R^{-1} \subset Y \times X $ and defined by $ y R^{-1} x $, if $ x R y $. Starting from a relation, we can define two abstract simplicial complexes called Dowker complexes.
		
		\begin{definition}\label{dfnDowK}
			Let $ R \subset X \times Y $ be a relation and $ K_R $ be an abstract simplicial complex. $ [ x_1, x_2, \ldots, x_n  ] \in K_R $ if and only if $ \exists y \in Y $ such that $ x_i R y $ for all $ i = 1, 2, \ldots, n $. 
		\end{definition}
		
		We have an analogous construction with $ Y $ as the set of vertices.	
		
		\begin{definition}\label{dfnDowL}
			Let $ R \subset X \times Y $ be a relation and $ L_R $ be an abstract simplicial complex. $ [ y_1, y_2, \ldots, y_n ] \in L_R $ if and only if $ \exists x \in X $ such that $ x R y_j $ for all $ j = 1, 2, \ldots, n $.
		\end{definition}
		
		The next theorem links the homotopy of $ \vert K_R \vert $ and $ \vert L_R \vert $.		
		
		\begin{theorem}[Dowker's Theorem]\label{thmDowHom}
			Let $ R \subset X \times Y $ be a relation. $ \vert K_R \vert  $ and $ \vert L_R \vert $ are homotopy equivalent.
		\end{theorem}

		The first version of Theorem (\ref{thmDowHom}) was done by C. Dowker in 1952 \cite{arDowkerCmp}. He defined the Dowker complex and shows that they are isomorphic in the \v{C}ech homology. In 1995, Björner \cite{bjorner1995} strengthened the result and showed the Theorem (\ref{thmDowHom}). In 2023, a different and interesting approach of the Dowker's Theorem \cite{arRectDow}. He showed a proof by using the rectangle complex. The Dowker Theorem has many applications in social studies with a method called Q-analysis develop by Atkin \cite{arQanalAtkin, arQanalAppSoc} and the theory of formal concept \cite{boFormalConcept}.

		We can define morphism between relations. Let $ R_1 \subset X_1 \times Y_1 $ and $ R_2 \subset X_2 \times Y_2 $ be relations. We say that $f : (X_1, Y_1, R_1) \to (X_2, Y_2, R_2) $ is a relation morphism if $ f = (f_1, f_2) $ is a pair of maps $ f_1 : X_1 \to X_2 $ and $ f_2 : Y_1 \to Y_2 $ such that for $ x \in X_1 $ and $ y \in Y_1 $, if $ x R_1 y $ implies that $ f_1(x) R_2 f_2(y) $. We obtain a category with relation as objects and relation morphism as morphisms.	For more information and results on functoriality and categorification of the Dowker theorem, we refer to \cite{arCoshefDow, arDowkerDuality, arFuncDowTheoAsNet, arFuncDowTheoNerves}.
		
		We need one more definitions on relations that we use in our algorithm.
		
		\begin{definition}\label{dfnColIre}
			A relation $ R \subset X \times Y $ is column irreducible, if for any $ y_1 $ and $ y_2 \in Y $, then $ R^{-1}(y_1) \not\subset R^{-1}(y_2) $.
		\end{definition}	
		
		If a relation $ R $ is column irreducible, we have a bijection between toplexes of $ K_R $ and elements of $ Y $.		
		
	\subsection{Gluing Theorem}	
		In our main proof, we use the Gluing Theorem. We cite two theorems from \cite{barmak2011algebraic}.
		
		\begin{theorem}[\protect{\cite[Th. A.2.4]{barmak2011algebraic}}]\label{thmA24}
			If $ A $ is a subcomplex of a CW-complex $ X $, the inclusion $ A \hookrightarrow X $ is a close cofibration.
		\end{theorem}
		
		There exist multiple versions of the Gluing Theorem. We use the following version. 
		
		\begin{theorem}[\protect{\cite[Th. A.2.5]{barmak2011algebraic}}]\label{thmA25}
			Suppose that the following diagram is a pushout of topological spaces
			\begin{center}
			\begin{tikzcd}
			 A \arrow[hookrightarrow]{d} \arrow[rightarrow]{r} & Y \arrow[rightarrow]{d} & \\
			 X \arrow[rightarrow]{r} & Z

			\end{tikzcd}
		\end{center}
		in which $ A \hookrightarrow X $ is a closed cofibration and $ A \to Y $ is a homotopy equivalence. Then, $ X \to Z $ is a homotopy equivalence.
		\end{theorem}			
	
	\section{Reduction step}\label{secRedStep}
	
		In this section, we show a way to reduce a simplicial complex by a simple method of conning. The main idea is to add a dummy vertex with a simplicial cone and we remove some vertices by applying a sequence of strong collapses. In this way, we remove vertices and simplices. We show it for general simplicial complexes. We argue that the notation of relation and Dowker complex are well suited for the reduction step.
				
		The first step of the reduction is to add a dummy vertex and a simplicial cone to the initial simplicial complex. We use the Gluing Theorem.
	
		\begin{lemma}\label{lemDummyVer}
			Let $ K $ be a finite simplicial complex and a finite set of vertices $ A = \{  x_1, x_2, \ldots,$ $ x_n \} $. We consider the subcomplex $ L = \cup_{i=1}^n \overline{St(x_i)} $. We add a dummy vertex $ z $ and $ K' = K \cup( z * L )$. If $ L $ is contractible, then $ \vert K \vert $ and $ \vert K' \vert $ are homotopically equivalent.
		\end{lemma}			
		\begin{proof}
			We consider the following diagram : 
			\begin{center}
				\begin{tikzcd}
			 \vert L \vert \arrow[hookrightarrow]{d} \arrow[hookrightarrow]{r} & \vert z*L \vert \arrow[hookrightarrow]{d} & \\
			 \vert K \vert \arrow[hookrightarrow]{r} & \vert K' \vert
				\end{tikzcd}.
			\end{center}	
			
			The diagram is a pushout because we have $ (z*L)  \cup K = K'$ and $ (z * L) \cap K = L $. Moreover, $ K $ and $ z*L $ are closed in $ K' $. We have that $ L $ is a subcomplex of $K$. By Theorem (\ref{thmA24}), we obtain that $ \vert L  \vert \hookrightarrow \vert K \vert $ is a closed cofibration. $L $ and $ z*L $ are contractible and this implies that $ \vert L \vert \hookrightarrow \vert z*L \vert $ is a homotopy equivalence. By the Gluing Theorem (\ref{thmA25}), we have that $ \vert K \vert$ and $ \vert K' \vert $ are homotopically equivalent.	

 		\end{proof}			
	
		The second step is to remove all the vertices in $L$ that are now dominated from the new vertex $z$.
	
		\begin{lemma}\label{lemRemoveVert}
			Let $ K $ be a finite simplicial complex and a finite set of vertices $ A = \{ x_1, x_2, \ldots,$ $ x_n \} $. We consider the subcomplex $ L = \cup_{i=1}^n \overline{St(x_i)}$. If there exists a $ z $ such that $ z*L $ is a subcomplex of $ K $, then $ K $ is homotopy equivalent to $ K \setminus A $.
		\end{lemma}	
		\begin{proof}
	    The vertex $ z $ dominate each $  x_i \in A$ because $ \overline{St(x_i)} \subset z * L $. Therefore, we can apply a finite sequence of strong collapses on each vertex in $ A $. $K$ is homotopy equivalent to $ K \setminus A $.
		\end{proof}
		
		By combining the two previous Lemmas, we get our reduction method stated as follows.	
		
		\begin{theorem}\label{thmReduceStep}
			Let $ K $ be a finite simplicial complex and  $ A = \{ x_1, x_2, \ldots,$ $ x_n \} $ a finite set of vertices. We consider the subcomplex $ L = \cup_{i=1}^n \overline{St(x_i)} $. If $ L $ is contractible, then $ K $, $ K' = K \cup (z*L) $ and $ K' \setminus A = K'' $ are homotopically equivalent.
		\end{theorem}	
		
		In Lemma (\ref{lemDummyVer}), we have an inclusion $ i : K \hookrightarrow K' $ which is an homotopy equivalence. In Lemma (\ref{lemRemoveVert}), we apply a finite sequence of strong collapses for each vertex in $A$. For each strong collapse, we have a retraction map. We compose all these retractions and we obtain a map $ r : K' \to K' \setminus A $ where all vertices in $A$ are mapped to $ z $ and the other vertices to itself. Finally, our reduction method is to apply $ r \circ i $.

		 Theorem (\ref{thmReduceStep}) reduced the initial simplicial complex because we add one vertex and remove $n$ vertices if $ n > 1 $. Moreover, the number of simplices is also reduced by at least $ n - 1 $ simplices.

	\subsection{Computation of Reduction Step with a Relation}
 	
 		We want to take advantage of relations and Dowker complexes to compute our reduction method.
 		
 		We use a binary matrix to define a relation. From a relation $ R \subset X \times Y $ such that $ X = \{ x_1, x_2, \ldots, x_m\} $ and $ Y = \{y_1, y_2, \ldots, y_n  \} $, we define the associated matrix $ M_R $ where $ m_{ij} = 1 $, if $x_i R y_j $ otherwise $ m_{ij} = 0 $. The binary matrix is our only data structure.  

		We can easily construct a relation $ R $ such that a simplicial complex $ S $ is equal to $K_R$. We take $ X $ the set of vertices of $ S $ and $ Y $ the set of toplexes $ S $. We have $ x R y $, if $ x \in Y $. We obtain that $ K_R = S $. Therefore, the relation reduction method can reduce any simplicial complex.
		
		In general, those matrices may have a lot of zeros. For a simplicial complex of dimensions $ d $, the number of ones in a column is bounded by $ d + 1 $. The dimension of a simplicial complex is way smaller than the number vertices and we obtain that the binary matrix is often sparse. Our structure is simple and memory efficient.
 		
		With this notation, we can easily compute the simplicial cone of $ z*L $ where $L$ is a set of toplexes from $K$. We add a new vertex $  z$ to obtain $ X' = X \cup \{ z \} $. The new relation is $ R' \subset X' \times Y $ and we set $ R'(z) = L $ and $ R'(x) = R(x) $ for $ x \neq z $. We obtain $ K_{R'} = K_R  \cup (z*L) $. 
 		
		An important computation we need to do is $ \overline{St(x)} $. This submatrix $ M_{R_{St}} $ is easily computed by considering the relation $ R_{St} = X \times R(x) $. For $ \cup_{x \in A} \overline{St(x)} $, we take the union which is $ R_{St} = X \times \cup_{x \in A} R(x) $. 
 		
		\begin{lemma}
			We have $ K_{R_{St}} = \cup_{x \in A} \overline{St(x)} $.
		\end{lemma} 
		\begin{proof}
			Let $ \sigma = [ x_1', x_2', \ldots, x_m' ] \in K_{R_{St}} $. There exists $ y \in \cup_{x \in A} R_{St}(x)  $ such that $ x_j' R_{St} y $ for all $j$. If there is a $ x_j \in A $, then $ \sigma \in \cup_{x \in A} \overline{St(x)} $. Otherwise, let $ \tau $ be the toplex with all the vertices in $ R_{St}^{-1}(y) $. We have $  \sigma \leq \tau $ and $ R_{St}^{-1}(y) \cap A \neq \emptyset $. Therefore, $ \sigma \in \cup_{x \in A} \overline{St(x)} $. 
			
			Let $ \sigma = [x_1', x_2', \ldots, x_m'] \in \cup_{x \in A}\overline{St(x)} $. There exist a toplex $ \tau \in \cup_{x \in A}\overline{St(x)} $ such that $ \sigma \leq \tau $, and there exist an $ x \in A$ such that $ x \in \tau $. By definition of $ K_R $, there exists a $ y \in R(x) $ such that $ x_j' R y$ for all $j$. By definition of $ R_{St} $ and $ y \in R(x) $, we also have that $ x_j' R_{St} y$. We obtain $ \tau \in K_{R_{St}} $ and this implies $ \sigma \in K_{R_{St}} $.
		\end{proof}
 		
		The second part of our reduction method is to remove vertices. We remove the row associated to $x_i$ in $ M_R $.

		In summary, our reduction method in matrix notation goes as follows.  We compute the submatrix associated to $ R_{St} = A \times \cup_{x \in A} R(x) $. We verify if $K_{R_{St}}$ is contractible. If this is true, then we add a new vertex $ z $ which add a new row in the matrix where $ R'(z) = \cup_{x \in A} R(x) $. Finally, remove all rows in the matrix $ M_{R'}$  that is associated to a vertex in $ A $. 
		
		There is still a major issue. The verification of $ L = \cup_{x \in A} \overline{St(x)}$ is contractible. We could directly compute the homotopy or the homology of $ L $. But this goes against our idea of using only relation to do our reduction method. Instead, we verify if it is strong collapsible. This work is already done in \cite{strongColPers}. We adapt their result with the notation of relation.
		
		\begin{lemma}\label{lemStrRed}
			Let $ R \subset X \times Y $ be a relation. Suppose there exists $ x_1, x_2 \in X $ such that $ R(x_1) \subset R(x_2) $. We consider $ R' \subset ( X \setminus \{ x_1 \} ) \times Y $ where $ R'(x) = R(x) $ for $ x \in X \setminus \{ x_1 \} $. Then, $\vert K_{R} \vert $ is homotopic to $\vert K_{R'} \vert $.
		\end{lemma}
		\begin{proof}
			$ R(x_1) \subset R(x_2) $ implies that for all $ y \in Y $ such that $ x_1 R y $, then $ x_2 R y $. This implies that $ L_{R} = L_{R'} $. By Dowker Theorem, we have that $ \vert K_R \vert $ is homotopic equivalent to $ \vert K_{R'} \vert $.
		\end{proof}
					
		A simple algorithm is to remove all rows that satisfies the Lemma (\ref{lemStrRed}). After that we transpose the matrix, we repeat the process until there is no more reduction possible. The resulted simplicial complex has the same strong homotopy type of the initial simplicial complex \cite{strongColPers}. Therefore, if the simplicial complex is reduce to a point, then we obtain that the simplicial complex is strong collapsible and this implies that it is homotopically trivial.

		The simplex reduction method from Lemma (\ref{lemStrRed}) is a special case of Theorem (\ref{thmReduceStep}). This is when $n = 2 $ and we obtain that $ R(z) = R(x_1) \cup R(x_2) = R(x_2) $, because $ R(x_1) \subset R(x_2) $. We only need to remove the row associated to $x_2$.
			
		\begin{example}
			We apply the Theorem (\ref{thmReduceStep}) on Figure \ref{figIntroCmpa}. First, we compute the relation $R$  such that $ K_R = S $. We choose the set of vertices $ A = \{ x_3, x_4 \} $. We compute $L = \overline{St(x_3)} \cup \overline{St(x_4)} $ which is the matrix $ M_{R_{St}} $. $ K_{R_{St}} $ is strong collapsible. We add a new row $z$ with the value $ R(z) = R(x_3) \cup R(x_4) $ to the matrix $ M_R $ and we obtain $ M_{R'} $. We remove the row $ x_3 $ and $ x_ 4$ of the matrix $ M_{R'} $ to finally obtain the reduced matrix $ M_{R''} $. At Figure (\ref{figIntro}), we see the different step of the reduction and we have the different matrices below.
		
			\begin{align*}
			M_R & = \begin{bmatrix}
				1 & 1 & 0 & 0 & 0 & 0 \\
				1 & 0 & 1 & 0 & 0 & 0 \\
				0 & 1 & 1 & 1 & 0 & 0 \\
				0 & 0 & 1 & 1 & 1 & 0 \\
				0 & 0 & 0 & 1 & 0 & 1 \\
				0 & 0 & 0 & 0 & 1 & 1
			\end{bmatrix} \quad  M_{R_{St}} = \begin{bmatrix}
				1 & 0 & 0 & 0 \\
				0 & 1 & 0 & 0 \\
				1 & 1 & 1 & 0 \\
				0 & 1 & 1 & 1 \\
				0 & 0 & 1 & 0 \\
				0 & 0 & 0 & 1
			\end{bmatrix} \\
			 M_{R'} & = \begin{bmatrix}
				1 & 1 & 0 & 0 & 0 & 0 \\
				1 & 0 & 1 & 0 & 0 & 0 \\
				0 & 1 & 1 & 1 & 0 & 0 \\
				0 & 0 & 1 & 1 & 1 & 0 \\
				0 & 0 & 0 & 1 & 0 & 1 \\
				0 & 0 & 0 & 0 & 1 & 1 \\
				0 & 1 & 1 & 1 & 1 & 0 	
			\end{bmatrix} \quad	 M_{R''} = \begin{bmatrix}
				1 & 1 & 0 & 0 & 0 & 0 \\
				1 & 0 & 1 & 0 & 0 & 0 \\
				0 & 0 & 0 & 1 & 0 & 1 \\
				0 & 0 & 0 & 0 & 1 & 1 \\
				0 & 1 & 1 & 1 & 1 & 0 	
			\end{bmatrix}
			\end{align*}			 							
		
		\end{example}

	\section{How to construct a relation from various data types}\label{secData} 			
	
		The main application of the relation reduction method is to reduce the Dowker complex $K_R$ associated to a relation. This can be used to simplify the simplicial complex and to preserve topological features. In this section, we will discuss different data structures that we can define by a relation. 
		 		
 		A common structure used in computational topology is simplicial complex. We remind the construction of relation for simplicial complex. Let $ R $ be a relation and $S$ a simplicial complex. We take $ X $ the set of vertices of $ S $ and $ Y $ the set of toplexes $ S $. We set $ x R y $, if $ x \in y $. We obtain that $ K_R = S $.

		We can generalize it for covers of a topological space $X$. We suppose that the covers have the finite intersection property. From a cover $N$, we construct a simplicial complex $ C(N) $ where $n_i$ is a vertex from each set of $N$ and a simplex $ [ n_1, n_2, \ldots n_m ]  \in C(N)$, if $ n_1 \cap n_2 \cap \ldots \cap n_m $ is non-empty. There are different versions of the Nerve Theorem. We cite the version of Borsuk.
		
		\begin{theorem}[\cite{arNerveBorsuk}]
			Let $ N $ be a cover which satisfies the condition that any intersection of elements of $N$ is either empty or contractible, then $ \vert C(N) \vert $ is homotopy equivalent to $ X $. 
		\end{theorem}
		
		Let $ \mathcal{N} $ be a cover of a topological space $X$ that satisfies the Nerve Theorem and $ P(\mathcal{N}) $ the power set of $ \mathcal{N} $. The relation $R$ is a subset of $ \mathcal{N} \times P(\mathcal{N})$, and if $ n \in \mathcal{N} $ and $ p \in P(\mathcal{N}) $ such that $ n \in p $ and $ \cap_{n_i \in p} n_i $ is non-empty, then $ n R p $.
		
		\begin{lemma}\label{lemRelCover}
			Let $ X $ be a topological space and $ \mathcal{N} $ a cover of $ X $ that satisfies the Nerve Theorem. Let $R \subset \mathcal{N} \times P(\mathcal{N}) $ be the relation defined above. We have that $ X $ is homotopically equivalent to $\vert K_R \vert $.
		\end{lemma}
		\begin{proof}
			We use the Nerve theorem with $ \mathcal{N}$ and we obtain the complex $ C(\mathcal{N}) $ that is homotopic equivalent to $X$. We show that $ C(\mathcal{N}) = K_R $. Let $\sigma = [n_1, n_2, \ldots, n_m] \in C(\mathcal{N}) $ if and only if $ \cap_{i=1}^m n_i \neq \emptyset $. The set $ p \in P(\mathcal{N})$ is $ \{ n_1, n_2, \ldots, n_m \} $. We obtain that $ n_i R p $ for each $i $. We obtain $ \sigma \in K_R $. The proof is similar in the other direction.
		\end{proof}
		
		This idea of relation of vertices and toplexes can be extended to convex polytope complex. This is a more general class of complex than simplicial complex and cubical complex.
		
		\begin{definition}[\cite{boPolytopes}]
			A convex polytope complex $ \mathcal{C} $ is a finite collection of convex polytopes in $ \mathbb{R}^d $ such that :
			\begin{itemize}
				\item the empty polytope is in $ \mathcal{C} $,
				\item if $ P \in \mathcal{C} $, then all the faces of $ P $ are also in $ \mathcal{C} $,
				\item the intersection $ P \cap Q $ of two polytopes $ P, Q \in \mathcal{C} $ is a face both of $ P $ and $ Q $ or empty.
			\end{itemize}
		\end{definition}	
		
		We do the same construction as the simplicial case. Let $ P $ be a convex polytopal complex, $ P_0 $ be the  set of vertices and $ P_{max} $ be the set of top dimensions polytopes. We construct $ R \subset P_0 \times P_{max} $ such that $ x R p $ if $ x \in p $.	
		
		\begin{lemma}
			Let $ P $ be a convex polytopal complex and $ R $ the relation defined above. Then, $ P $ is homotopy equivalent to $\vert K_R \vert$.
		\end{lemma}
		\begin{proof}
			Let $\mathcal{N}$ be the cover of $P$ with maximal dimension polytopes. All elements of $N$ are convexes so also contractible. We can use the Nerve Theorem and we obtain the complex $ \vert C(\mathcal{N}) \vert $ which is homotopy equivalent to $ P $. We want to show that $ C(\mathcal{N}) = K_R $.
			
			Let $ \sigma = [n_1, n_2, \ldots, n_m] \in C(\mathcal{N}) $. This implies that $ \cap_{i=1}^m n_i \neq \emptyset $. There exists at least an $ x \in P_0 $ such that $ x \in \cap_{i=1}^m n_i $ and $ x R n_i $ for all $i$. We have $ \sigma \in K_R $. The proof is similar in the other direction.
		\end{proof}		
		
		An advantage of our method is that it can be used with different data structures. If we can construct a relation that the Dowker complex is homotopically equivalent to the topological space, then we can use our reduction method.
		 			
	\section{Reduction Algorithm}\label{secRedAlgo}
	
		In this section, we present a simple algorithm by using the Theorem (\ref{thmReduceStep}) with two vertices. It takes as input a binary matrix $ M_R $ from a relation $R$ which is column irreducible. In output, we obtain a reduced relation $ R' $ in a binary matrix $ M_{R'} $. Our algorithm go as follows. We go through each vertex in the current reduced relation $ R' $. We choose a row of the matrix and compute $ A = \overline{St(\overline{St(x_i)})} $. If $ x_j \in A $ a vertex, then $ \overline{St(x_i)} \cap \overline{St(x_j)} $ is non-empty. All those vertices are in a priority queue $ pq\_Vert $. We do a simple partial ordering. If $ x_j \in \overline{St(x_i)} $ and $ x_k \notin \overline{St(x_i)} $, then we place $ x_j $ before $ x_k $. For each $ x_j \in  pq\_Vert $ such that $ i < j $, we compute $ \overline{St(x_i)} \cup \overline{St(x_j)} $ and verify if it is contractible. If it is true, then we can apply our reduction. If it is false, then we take the next vertex in $ pq\_Vert $ until there is no more already processed vertices. For the reduction step, we add a new vertex $ z $. We compute the row $ R'(z) = R'(x_i) \cup R'(x_j) $. We remove rows $ x_i $ and $ x_j $ and add the row $z$ at the end of the matrix $M_{R'}$. We apply column reduction on $M_{R'}$ because we might lose the property column irreducible after removing the row. We reduce the number of vertices by one and we select the next row and we repeat the process until all rows are processed. We use the index of rows to detect if a row is already processed. 
 		
 		We obtain the main algorithm :		
 		
 	\begin{algorithm}[H]
	\caption{Reduction of Relation} 
	\begin{algorithmic}[1]
		\REQUIRE A binary matrix $ M_R $ where $ R $ is a relation with column irreducible.
		\ENSURE A reduce matrix $M_{R'}$.
		\STATE $n_1 \leftarrow $  number of rows in $ M_R $
		\FOR{ $ 0 \leq i \leq n_1 $}
		\STATE $ pq\_Vert \leftarrow $ vertices of $ \overline{St(\overline{St(x_i)})} $ and remove all $ x_j $ such that $ j \leq i $
		\FOR{$ x_j \in pq\_Vert $ }
		\STATE $ M_{R_{St}} \leftarrow \overline{St(x_i)} \cup \overline{St(x_j)} $
		\IF{$M_{R_{St}}$ is contractible}
		\STATE Compute $ row_z $ : $R(z) \leftarrow R(x_i) \cup R(x_j)$.		
		\STATE Remove $ row_i $ and $ row_j $ from $ M_R $.
		\STATE Add $ row_z $ at the end of $ M_R $.
		\STATE $ n_1 \leftarrow n_1 - 1 $
		\STATE $ R \leftarrow $ columnReduction($R(z)$)
		\STATE Go to line $3$ with the same $i$.
 		\ENDIF 
		\ENDFOR
		\ENDFOR
		\RETURN The reduced relation $R$.
	\end{algorithmic}\label{algReduc}
	\end{algorithm}	 	

We only need to do a single pass on rows of the matrix. Even if we apply a reduction step, we do not need to come back to already processed rows in $ R $.   
	
	\begin{lemma}\label{lemNotHomPt}
		Let $ R \subset X \times Y $ be a relation. We suppose there exists $ x_i $ such that $ \vert  \overline{St(x_i)}\cup \overline{St(x_{j_1})} \vert $ is not contractible for an $ x_{j_1} \in X $. Let $ R' $ a reduced relation of $ R $ where we remove $ x_{j_1} $ and $ x_{j_2} $, and we add $z $. Then, $\vert \overline{St(x_i)} \cup \overline{St(z)} \vert$ is not contractible in $ K_{R'} $.
	\end{lemma}	
	\begin{proof}
		By our reduction, we add $ z * ( \vert \overline{St(x_{j_1})} \cup \overline{St(x_{j_2})} \vert ) $ to $ K $ to obtain $ K' $. We consider $ \vert \overline{St(x_i)} \cup \overline{St(x_{j_1})} \vert = L'$ in $ K' $. We also have that $ z $ dominate $ x_{j_1} $ in $L'$ and also $ x_{j_2} $, if $ x_{j_2} \in L' $. We can strong collapse $ x_{j_1} $ and $ x_{j_2} $ in $L'$ to obtain $\vert \overline{St(x_i)} \cup \overline{St(z)} \vert $ in $ K_{R'} $. We have that $ L'$ in $ K' $ is not contractible. Therefore, $\vert \overline{St(x_i)} \cup \overline{St(z)} \vert$ is also not contractible.  
	\end{proof}
	
		 The previous Lemma (\ref{lemNotHomPt}) means that we do not need to come back to already process vertices after we added a new vertex. In our Algorithm 1, we used the index of rows $ x_j $ to determine that a row is already processed in $ Pq\_vert $. Therefore, we remove all $ x_j $ such that $ j \leq i $ in $ Pq\_vert $.

	For the column reduction method, we do not need to verify every column of $ R' $, but only for the columns of $ R'(z)$. 
	
	\begin{lemma}\label{lemRedCol}
		Let $ R \subset X \times Y $ be an irreducible relation and $ y_1, y_2 \in Y $. Consider $ R' $ a reduced relation of $R$ where we removed $ x_1 $ and $ x_2 $ and we added the vertex $z$. If $ y_1 \not\in R'(z) $ or $ y_2 \not\in R'(z) $, then $ R'^{-1}(y_1) \not\subset R'^{-1}(y_2)$ and $ R'^{-1}(y_1) \not\supset R'^{-1}(y_2) $. 
	\end{lemma} 		
 	\begin{proof}
 		Suppose that $ y_1 \not\in R'(z) $. If $ y_2 \not\in R'(z) $, then it is trivial. Because $ R $ is column irreducible and we have $ R'(y_1) = R(y_1) $ and $ R'(y_2) = R(y_2) $. 
 		
		If $ y_2 \in R'(z) $, we have $ z \in R'^{-1}(y_2) $ and $ z \not\in R'^{-1}(y_1) $. This implies that $ R^{-1}(y_1) \not\subset R^{-1}(y_2) $. There exists an $ x \in R^{-1}(y_1) = R'^{-1}(y_1)$ such that $ x \not\in R^{-1}(y_2) $, because $ R $ is column irreducible. But $ x \neq z $. Therefore, $ R'^{-1}(y_1) \not\supset R'^{-1}(y_2) $
	
 	\end{proof}
 		
	There are two cases when one can remove some columns of the relation after a reduction step. 
		
	The first case is when two toplexes become the same one after a reduction. Let $ \sigma = [x_i, x_1, x_2, \ldots, x_n] $ and $ \tau = [x_j, x_1, x_2, \ldots, x_n] $ be toplexes. After the reduction step of $ x_i $ and $ x_j $ to $ z $, toplexes $ \sigma $ and $ \tau $ become $ \sigma' = [ z, x_1, x_2, \ldots, x_n ] = \tau' $. Therefore, $ R' $ is no more column irreducible, because we have a duplicate column. We remove one of the two columns. 
		
	The second case is when a toplex become a face after a reduction. Let $ \sigma = [ x_i, x_j, x_1, x_2,$ $ \ldots, x_n ] $ be a toplex. After the reduction step of $ x_i $ and $ x_j $ to $ z $, we have $ \sigma' = [ z, x_1, x_2, \ldots, x_n ] $. There is a reduction in the dimension of $ \sigma' $. There might exist another toplex $ \tau $ such that $ \tau > \sigma' $. If this is the case, we remove the column associated to the toplex $ \sigma' $. 	

	\subsection{Contractibility of $L$ and sorting of $ Pq\_vert $}	
 		
 	A main step of the algorithm is to verify that $ \overline{St(x_i)} \cup \overline{St(x_j)} $ is contractible. We use the work of \cite{strongColPers} to verify that it is strongly collapsible. This can be costly. Therefore, we hope that the shapes of the matrix associated to $ \overline{St(x_i)} \cup \overline{St(x_j)} $ are small. But even if at the beginning submatrices are small, it grows during the algorithm. But it will never be bigger than the relation matrix.
		
		Let $ \delta_{x_i} $ be the number of vertices in $ \overline{St(x_i)} $ and $ \epsilon_{x_i} $ be the number of toplexes in $ \overline{St(x_i)} $. Suppose that $ R $ is column irreducible. Therefore, the submatrix representing $ \overline{St(x_i)} $ has the shape $ \delta_{x_i} \times \epsilon_{x_i} $. We are interested in the shape of the matrix $ \overline{St(x_i)} \cup \overline{St(x_j)} $. The number of rows is $ \delta_{x_i} + \delta_{x_j} - \delta_{x_i \cap x_j} $ and the number of columns $ \epsilon_{x_i} + \epsilon_{x_j} - \epsilon_{x_i \cap x_j} $ where $ \delta_{x_i \cap x_j} $ is the number of shared vertices in $ \overline{St(x_i)} $ and $ \overline{St(x_j)} $, and $ \epsilon_{x_i \cap x_j} $ is the number of shared toplexes in $ \overline{St(x_i)} $ and $ \overline{St(x_j)} $.
		
		Now, let see how $ \delta_{x_k} $ is changed during the algorithm. Let $ \delta_{x_k}^{(l)} $ denote $ \delta_{x_k} $ after $l$ reductions. Suppose that $ \overline{St(x_i)} \cup \overline{St(x_j)} $ is contractible. We apply the reduction by removing $ x_i $ and $ x_j $, and we add a new vertex $ z $. There are 4 different cases. These equations are simple and we let the proof to the reader.
		
		\begin{equation*}
			\begin{cases}
				\delta_{x_k}^{(l+1)} = 0 & k = i,j \\
				\delta_{x_k}^{(l+1)} = \delta_{x_i}^{l)} + \delta_{x_j}^{(l)} - \delta_{x_i \cap x_j}^{(l)} & x_k = z \\
				\delta_{x_k}^{(l+1)} = \delta_{x_k}^{(l)} - 1 & \text{If } x_i, x_j \in \overline{St(x_k)} \\
				\delta_{x_k}^{(l+1)} = \delta_{x_k}^{(l)} & \text{Otherwise}
			\end{cases}.
		\end{equation*}
		
		Let see how $ \epsilon_{x_k}^{(l)} $ is affected by a reduction. We also have four cases :
		
		\begin{equation*}
			\begin{cases}
				\epsilon_{x_k}^{(l+1)} = 0 & k = i, j \\
				\epsilon_{x_k}^{(l+1)} = \epsilon_{x_i}^{(l)} + \epsilon_{x_j}^{(l)} - \epsilon_{x_i \cap x_j}^{(l)} - \# faces - \frac{\# same}{2} & x_k = z \\
				\epsilon_{x_k}^{(l+1)} = \epsilon_{x_k}^{(l)} - \# faces - \frac{\# same}{2} & \text{If } x_i, x_j \in \overline{St(x_k)} \\
				\epsilon_{x_k}^{(l+1)} = \epsilon_{x_k}^{(l)} & \text{Otherwise}
			\end{cases}.
		\end{equation*}				
		
		where $ \# faces $ is the number of toplexes which become a face and $ \# same $ is the number of toplexes which identified to the same toplex after a reduction in $ \overline{St(x_i)} \cup \overline{St(x_j)} $.
		
		In cases 1, 3 and 4 of $ \delta_{x_i}^{(l)} $, we either reduce or stay equal. But, for the second case, $ \delta_{x_k}^{(l+1)} $ can grow fast. Let $ \epsilon_{max}^{(l)} =  \max_{x_i \in X} \epsilon_{x_i}^{(l)}  $ and $ \delta_{max}^{(l)} = \max_{x_i \in X} \delta_{x_i}^{(l)} $. Therefore, the biggest matrix that we need to test at iteration $l$ has its shape bounded by $ (2 \epsilon_{max}^{(l)}) \times (2 \delta_{max}^{(l)}) $. With the equation from the case of $ \epsilon_z $ and $ \delta_z $, we obtain two bounds for each reduction step : 
		
		\begin{gather}
		 	\delta_{max}^{(l+1)} \leq \delta_{max}^{(l)} + \delta_{max}^{(l)} = 2 \delta_{max}^{(l)} \leq \text{number of rows in the matrix of } R'^{(l+1)};	\\
		 	\epsilon_{max}^{(l+1)} \leq \epsilon_{max}^{(l)} + \epsilon_{max}^{(l)} = 2 \epsilon_{max}^{(l)} \leq \text{number of columns in the matrix of } R'^{(l+1)}
		\end{gather}
		
		where $ R'^{(l+1)} $ is  the relation after applying $ l+1$ reduction steps.
		
		In summary, when the new vertex $ z $ is added, it can double the maximum number of vertices and the maximum number of toplexes in a star of a vertex.
		 
		  Consider the worst case of merging. This happens when $ \overline{St(x_i)} $ and $ \overline{St(x_j)} $ intersect at a single vertex and $ \dim(\overline{St(x_k)}) > 2 $ for $ k = i,j $. We obtain that $ \delta_z =  \delta_{x_i} + \delta_{x_j} - 1 $ and $ \epsilon_z = \epsilon_{x_i} + \epsilon_{x_j} $. If we are not careful of how we choose $ x_i $ and $ x_j $, then $ \epsilon_{max} $ and $ \delta_{max} $ can grow exponentially . For our algorithm, we choose a simple method by giving a better priority to some vertices. We fix a $ x_i \in X $. Then, we verify the condition for each $ x_j \in \overline{St(x_i)} $. If none works, then we can choose the other remaining vertices in $ x_j \in \overline{St(\overline{St(x_i)}} $. We choose to have a simple sorting of $ Pq\_vert $ because it is easier to compute. By choosing $ x_j \in \overline{St(x_i)} $, we have that that they share at least a toplex and multiple vertices. Therefore, it slows the growth of $ \delta_{\max} $ and $ \epsilon_{\max} $. A more sophisticated sort can be used to reduce even more the growth of $ \delta_{\max} $ and $ \epsilon_{\max} $. But this kind of sorting takes more time to do. This is a trade off that someone needs to choose.  
	
 	\subsection{Time complexity}	
		
		We compute the time complexity in average and the worst case of the Algorithm (\ref{algReduc}).
		
		Let $R \subset X \times Y$ be a relation,  $n = \vert X \vert  $, $ m = \vert Y \vert $, $ \epsilon_{x_i} $ the number of toplexes in $ \overline{St(x_i)} $ and $ \delta_{x_i} $ the number of vertices in $ \overline{St(x_i)} $.

		We fix the relation $ R $. We need to verify that $  \overline{St(x_i)} \cup \overline{St(x_j)} = L $ is contractible. We know that if $ \overline{St(x_i)} \cap \overline{St(x_j)} = \emptyset $, then $ L $ is not contractible. For a fixed $ x_i $, we only need to   test for $ x_j \in \overline{St(\overline{St(x_i)})} $. Let $ d_{x_i} $ be the number of vertices in $ \overline{St(\overline{St(x_i)})} $. The number of different sub-complexes $ L $ is given by $c$ :
		\begin{equation}\label{eqComp}
			c = \frac{\sum_{i=1}^n (d_{x_i} - 1)}{2}.
		\end{equation}

		Let $ c $ associated to the relation $R$ and $ c' $ associate to the reduced relation $R'$, then $ c > c' $. From Lemma (\ref{lemNotHomPt}), after a reduction step, we do not need to verify already processed vertices. The number of comparisons, we do during our Algorithm \ref{algReduc}, is bounded by $c$ from the initial simplicial complex. We have that $ c $ is $ O(n^2) $ by bounding $ d_{x_i} $ by $n$ for all $i$.

		As said earlier, to verify that $ L $ is contractible, we used the algorithm of Boissonat and al. in \cite{strongColPers}. In short, we used the Lemma (\ref{lemStrRed}) on $ R' $ and $ R'^{-1} $ until there is no more reduction. The shape of our submatrix is $ (\delta_{x_i} + \delta_{x_j}) \times (\epsilon_{x_i} + \epsilon_{x_j}) $. The average time complexity to verify that $ \vert \overline{St(x_i)} \cup \overline{St(x_j)} \vert = L $ with our notation is $ O(d \epsilon(\delta_{max} + \epsilon_{max})) $ where $d$ is the dimension of the complex $ L $ and $ \epsilon $ the maximal number of toplexes adjacent to the same vertex in $L$. 
			
		Suppose that it takes constant time to compute $ \vert \overline{St(x_i)} \cup \overline{St(x_j)} \vert $, and to add and remove rows of a matrix. For the computing and partial sorting of $ pq\_Vert $, we also suppose that it is constant. 
		
		The worst case is when the simplicial complex $K$ is the boundary of an $n$-simplex. For each $ x_i \in X $, we have $ x_j \in \overline{St(x_i)} $ for all $x_j \in X $. Therefore, $ \delta_{x_i} = n $. We also have that $ \vert \overline{St(x_i)} \cup \overline{St(x_j)} \vert = \vert K \vert $ are not contractible for all $ x_i, x_j \in X $ with $ i \neq j $. From Equation (\ref{eqComp}), the value $c$ becomes :
		
		\begin{gather*}
			c = \frac{\sum_{i=1}^n (\delta_{x_i} - 1)}{2} = \frac{\sum_{i=1}^n (n - 1)}{2} = \frac{n^2 - n}{2} = \frac{n(n-1)}{2}.
		\end{gather*}
		
		 All submatrices are the same and they are equal to $M_R$. In summary, the time complexity of the worst-case scenario is $ O(n^2n(n^2n +m^2m)) = O( n^3(n^3 + m^3)) $. The worst case is very bad. But usually, we have $ d << n$ and $ \epsilon_{x_i} << m $. 		
		
	\section{Experimental results}\label{secExpRes}
	
		In this section, we will discuss about the efficiency of the reduction of the Algorithm \ref{algReduc}. We test our algorithm on different datasets. They are Small sphere, Big sphere, Small torus, Big torus, Space Shuttle \cite{arMultiDimMorse}, Tie Fighter \cite{arMultiDimMorse}, X-wing \cite{arMultiDimMorse}, Space Station \cite{arMultiDimMorse}, Bunny \cite{webSiteStanfordBunny} and Armadillo \cite{webSiteStanfordArmadillo}. All datasets are surfaces embedded in $ \mathbb{R}^3 $.
		 
	
		Table \ref{tabExpResults} has the following columns in order : name of the dataset, the number of vertices, and the number of toplexes in the initial complex, Betti number of dimensions $0$, $1$ and $2$, the number of vertices, and the number of toplexes after applying our reduction algorithm.  
			
		Let's discuss the result at Table \ref{tabDataExp}. We see for most of the datasets that there is a huge reduction in the number of toplexes and vertices. In particular, the Armadillo was able to be reduced back to the boundary of a $3$-simplex which is smallest simplicial complex homeomorphic to a sphere. But our method is less effective in the reduction if the Betti numbers are high. This makes sense, because they have more topological features that we need to preserve. This can be seen in datasets Space Station, X-wing and Tie Fighter. Our reduction method is efficient for the Bunny, Armadillo, Sphere and Torus because they have less topological features.
		 
	\begin{table}[H]	
		\centering
		\begin{tabular}{|c|c|c|c|c|c|c|c|}
			\hline
			Dataset & Vertices & Toplexes & $ \beta_0 $ & $ \beta_1 $ & $ \beta_2 $ & Red. Vert. & Red. Top. \\
			\hline
			Small Sphere 	&  8  	& 12	& 1 	& 0 	& 1 	& 4 	& 4 	\\
			Big Sphere 		& 482 	& 960	& 1 	& 0 	& 1 	& 4 	& 4 	\\
			Small Torus 	& 16  	& 32	& 1 	& 2 	& 1 	& 7 	& 14 	\\
			Big Torus 		& 1200 	& 2400	& 1 	& 2 	& 1 	& 9 	& 18 	\\
			Space Shuttle 	& 2376 	& 3952	& 5 	& 7 	& 0 	& 17 	& 18 	\\
			Tie fighter 	& 2014 	& 3827	& 18 	& 148 	& 6 	& 159 	& 403 	\\
			X-wing 			& 3099 	& 6076	& 18 	& 50 	& 13 	& 106 	& 184 	\\ 
			Space Station 	& 5749 	& 10237	& 110 	& 116 	& 39 	& 397 	& 496 	\\ 
			Bunny  			& 34834	& 69451 & 1 	& 4 	& 0 	& 7 	& 12 	\\
			Armadillo  		& 172974 & 345944 & 1 	& 0 	& 1 	& 4 	& 4 	\\
			\hline
		\end{tabular}
		\caption{Experimental results on different datasets.}\label{tabExpResults}
	\label{tabDataExp}
	\end{table}

	\section{Conclusion}
 		
 		We have shown a new reduction method for a simplicial complex $K$ which summarize as follows. First, we choose a set of vertices $ A $ such that $ L = \cup_{x \in A } \overline{St(x)} $ is contractible. After, we added a new vertex $z$ and we add $ z * L $ to $K$ by applying the Gluing Theorem. Now, all vertices in $A$ are now dominated by $ z $. Then, we strong collapse each vertex in $A$ and obtain a new reduced simplicial complex. We argued that relations with their Dowker complexes are well suited for this reduction. Because the relation is a binary matrix and often sparse. The number of rows is the number of vertices and the number of columns is the number of toplexes. We have the sparsity of the binary matrix, if the dimension of the simplicial complex is way smaller than the number of vertices. All the operations we need to apply for the reduction are simple except one. This is to verify that $ L $ is contractible. We use the algorithm of \cite{strongColPers} to verify if $L$ is strong collapsible  because it works with our data structure of binary matrix.
 		
 		We have $ O(n^2)$ different subcomplexes $L$ to verify that it is contractible where $ n $ is the number of vertices. The number of operations to verify that $L$ is strong collapsible is $O(d\epsilon(n^2 + m^2)) $ where $ d $ is the dimension, $ \epsilon $ the maximum toplex adjacent to a vertex in $L$, $ m $ the number of toplexes in $ L $, and $ n $ the number of vertices in $L$. If a reduction step is done, we can remove some columns which the cost is $ O(\epsilon^2 d) $.
 		
 		They are multiple advantages in using this approach of reduction. Our reduction method generalizes existing methods based on strong collapse and edge collapse. Our method is memory efficient because the only data structure needed is a binary matrix which is often sparse. In general, the number of toplexes is a way smaller than the number of simplices.
 		
 		Moreover, we can construct the relation $R$ for different type data. We showed a simple way to construct a relation from vertices and toplexes. But in general, if we can use the Nerve Theorem, then we can easily construct a relation where $ \vert K_R \vert $ is homotopic to the topological space.
 		
		A major problem is to verify that $L$ is contractible. It is still a bottleneck in our computation. Even if the submatrix of $ M_{R_{St}} $ is quite small, we do them $ O(n^2) $ times.


	\typeout{}
	\bibliography{biblio}

\end{document}